\documentclass[a4paper,10pt,reqno, english]{amsart}

\usepackage{amsmath,amssymb,amscd,amsthm,amsfonts}
\usepackage{graphicx, subfigure, tikz-cd, soul, verbatim}
\usepackage[pdfencoding=auto]{hyperref}
\usepackage{dsfont}
\usepackage{xcolor}
\usepackage[utf8]{inputenc}
\usepackage{microtype}
\usepackage[alphabetic]{amsrefs}
\usepackage[capitalize]{cleveref}
\usepackage{thm-restate, thmtools}

\newtheorem{theorem}{Theorem}[section]

\newtheorem{lemma}[theorem]{Lemma}
\newtheorem{proposition}[theorem]{Proposition}

\newtheorem{corollary}[theorem]{Corollary}

\newtheorem{fact}[theorem]{Fact}

\newtheorem{definition}[theorem]{Definition}

\newcommand{\rr}{\mathbb{R}}

\newcommand{\nn}{\mathbb{N}}
\newcommand{\ntptwo}{\text{NTP}_2}
\newcommand{\tptwo}{\text{TP}_2}
\newcommand{\nip}{\text{NIP}}

\newcommand{\bdn}{\text{bdn}}
\newcommand{\nipdp}{\text{NIP-dp}}
\newcommand{\niptp}{\text{NIP-tp}}
\newcommand{\rkdp}{\text{dp}}

\newcommand{\tp}{\text{tp}}

\def\Ind#1#2{#1\setbox0=\hbox{$#1x$}\kern\wd0\hbox to 0pt{\hss$#1\mid$\hss}
\lower.9\ht0\hbox to 0pt{\hss$#1\smile$\hss}\kern\wd0}

\def\ind{\mathop{\mathpalette\Ind{}}}

\def\notind#1#2{#1\setbox0=\hbox{$#1x$}\kern\wd0
\hbox to 0pt{\mathchardef\nn=12854\hss$#1\nn$\kern1.4\wd0\hss}
\hbox to 0pt{\hss$#1\mid$\hss}\lower.9\ht0 \hbox to 0pt{\hss$#1\smile$\hss}\kern\wd0}

\def\nind{\mathop{\mathpalette\notind{}}}

\title{Consequences of Dependent Dividing on Burden}
\author{Yuki Takahashi}
\address{University of California, Berkeley, CA 94720}
\email{yukit@berkeley.edu}
\date{October 31, 2025}
\subjclass[2020]{03C45}
\keywords{dependent dividing, burden, inp-rank, dp-rank, VC-density}
\thanks{}

\begin{document}
\begin{abstract}
If $T$ has dependent dividing, then the burden agrees with the dp-rank witnessed by NIP formulas.
We use this observation to prove that if $T$ has dependent dividing, then the burden is sub-additive. We also state a connection between the burden and the dual VC density. 
\end{abstract}
\maketitle
\section{Introduction}
The stable forking conjecture is one of the most important open conjectures in simple theories. In this paper, we investigate the stable forking conjecture's lesser-known cousin, the dependent dividing conjecture. 
The dependent dividing conjecture was first stated in \cite{Chernikov14} but has not been explored extensively. 
While stable forking is stated for simple theories, dependent dividing generalizes it to $\ntptwo$ theories. Here are the key definitions of $\ntptwo$ and dependent dividing. Note that throughout this paper, we do not distinguish between a singleton and a tuple.
\begin{definition}\label{def:tp2}
    A formula $\phi(x,y)$ has $\tptwo$ if there is an array $(b_{\alpha, i})_{\alpha,i<\omega}$ such that $\{\phi(x,b_{\alpha,i})\}_{i<\omega}$ is 2-inconsistent for every $\alpha<\omega$ and $\{\phi(x,b_{\alpha,f(\alpha)})\}_{\alpha<\omega}$ is consistent for any $f:\omega\to\omega$. Otherwise, we say that $\phi(x,y)$ is $\ntptwo$.
    A theory $T$ is $\ntptwo$ if every formula is $\ntptwo$.
\end{definition}

\begin{definition}\label{def:depdiv}
    We say that a theory $T$ has dependent dividing if given models $M, N$ with $M\preceq N$ and $p(x)\in S(N)$ dividing over $M$, there is an NIP formula $\phi(x;y)$ and $c\in N$ such that $\phi(x,c)\in p(x)$ and $\phi(x,c)$ divides over $M$.
\end{definition}

\begin{fact}[\cite{Chernikov14}, Proposition 4.14]\label{fact:depdiv-ntp2}
    If $T$ has dependent dividing (or even just $\ntptwo$ dividing), then $T$ is $\ntptwo.$ 
\end{fact}
The dependent dividing conjecture is the converse of the above fact: \textit{if $T$ is $\ntptwo$, then $T$ has dependent dividing}. We investigate the implications of dependent dividing on the burden, which is defined as follows:
\begin{definition}\label{def:bdn}
An inp-pattern (inp stands for independent partition) in $p(x)$ of depth $\kappa$ consists of $(b_{\alpha, i})_{\alpha<\kappa, i<\omega}$, $(\phi_\alpha(x,y_\alpha))_{\alpha<\kappa}$ and $k_\alpha <\omega$ such that 
\begin{itemize}
\item $\{\phi_\alpha(x,b_{\alpha, i})\}_{i<
\omega}$ is $k_\alpha$-inconsistent, for each $\alpha<\kappa$.
\item $\{\phi_\alpha(x,b_{\alpha, f(\alpha)})\}_{\alpha<
\kappa}\cup p(x)$ is consistent, for any $f: \kappa\to \omega$.
\end{itemize}
The burden of $p(x)$, denoted $\bdn(p)$, is the maximum (if it exists) of the depths of all inp-patterns in $p(x)$. 
If $\sup\{\lambda: \text{there is an inp-pattern in $p(x)$ of depth $\lambda$}\}=\kappa$ but there is no inp-pattern of depth $\kappa$, we say $\bdn(p(x))=\kappa_-$.
By $\bdn(a/C)$ we mean $\bdn(tp(a/C))$.
\end{definition}
The definition of burden was introduced in \cite{Adler07}, and it is sometimes referred to as the inp-rank. 
Burden is a notion of a rank that is suitable for $\ntptwo$ theories, since a theory $T$ is $\ntptwo$ iff $\bdn(a/C)<|T|^+$ for every tuple $a$ and set $C$ \cite[Lemma 3.2]{Chernikov14}. Fundamental properties of burden were established in \cite{Chernikov14}; readers should refer to \cite{Chernikov14} for a comprehensive introduction to $\ntptwo$ theories and burden. Since then, groups and fields in which every type has finite burden have been explored in various papers, such as 
\cite{johnson16}, \cite{DG17}, \cite{DG23}, \cite{DG25}, \cite{DG19}, \cite{DW19}, \cite{CS19}, \cite{Touchard23}, and \cite{Fujita25}.

The guiding principle of this paper is that if $T$ has dependent dividing, then the burden behaves similarly to the dp-rank, allowing us to transfer results in dp-rank to burden. 
In particular, assuming dependent dividing, the burden is equivalent to the dp-rank witnessed by NIP formulas, which we call the NIP dp-rank. 

For a notation, given some sequence $I=\langle a_i: i\in \mathcal{I}\rangle$ and some sets $B$ and $C$, we say that $I$ is \textit{NIP-indiscernible over $(B;C)$} if for every NIP formula $\phi(x_0,\ldots,x_n,y;{z})$ with $|x_0|=\cdots=|x_n|=|a_i|$, tuple $b\in B$ of length ${|y|}$, tuple $c\in C$ of length $|z|$, and $i_0<\ldots<i_n$ and $ j_0<\ldots<j_{n}$ from $ \mathcal{I}$, we have \[\vDash \phi(a_{i_0},\ldots a_{i_n},b; c)\leftrightarrow \phi(a_{j_0},\ldots, a_{j_n},b; c).\]
Note that any formula $\phi(x_0,\ldots,x_n,y)$ with no parameter variables is NIP, so if $I$ is NIP-indiscernible over $(B;C)$, then it is indiscernible over $B$.
We say that $I$ is \textit{NIP-indiscernible over $C$} if it is NIP-indiscernible over
$(\emptyset;C)$.
We say that a collection of sequences $\{I_\alpha:\alpha<\kappa\}$ is \textit{mutually NIP-indiscernible over $(B;C)$} if each $I_\alpha$ is NIP-indiscernible over $(BI_{\beta_{\beta\neq \alpha}};C)$.
Similarly, a collection of sequences $\{I_\alpha:\alpha<\kappa\}$ is \textit{mutually NIP-indiscernible over $C$} if each $I_\alpha$ is NIP-indiscernible over $(I_{\beta_{\beta\neq \alpha}};C)$.

\begin{definition}\label{def:nip-dprank}
    Let $p(x)$ be a (partial) type over $C$. 
    We define the NIP dp-rank of $p(x)$, denoted $\nipdp(p(x))$, to be the maximum (if it exists) of $\kappa$ for which there exist $d\vDash p(x)$ and $\{I_\alpha: \alpha<\kappa\}$,
    mutually NIP-indiscernible sequences over $C$ such that for all $\alpha<\kappa$, $I_\alpha$ is not NIP-indiscernible over $Cd$.
    We define $\nipdp(p(x))=\kappa_-$ similarly to how it is defined for the burden.
\end{definition}
Note the similarity of the NIP dp-rank to the dp-rank. The dp-rank of a partial type $p(x)$ over $C$, denoted $\rkdp(p(x))$, is the maximum (if it exists) of $\kappa$ for which there exist $d\vDash p(x)$ and $\{I_\alpha: \alpha<\kappa\}$, mutually indiscernible sequences over $C$, such that for all $\alpha<\kappa$, $I_\alpha$ is not indiscernible over $Cd$; $\rkdp(p(x))=\kappa_-$ is defined similarly. 
It is known that the burden agrees with the dp-rank in NIP theories. The definition of dp-rank was introduced in \cite{Shelah14}, and dp-rank is suitable for $\nip$ theories, since a theory $T$ is $\nip$ iff $\rkdp(a/C)<|T|^+$ for every tuple $a$ and set $C$ \cite[Observation 4.13]{Simon15}. Fundamental properties of the dp-rank were established in \cite{KOU11}, \cite{OS11}, \cite{KS14}. Readers should refer to \cite{Simon15} for a comprehensive introduction to NIP theories and dp-rank.

We now introduce the structure of the paper and the main theorems in each section. 
In \cref{sec:prelim}, we introduce the key definitions and establish the equivalence between the burden and the NIP dp-rank in theories with dependent dividing. 

In \cref{sec:subaddbdn}, we show that the burden is sub-additive in theories with dependent dividing. Note that the burden is already known to be sub-multiplicative in any theory (\cite[Corollary 2.6]{Chernikov14}). Since the burden is sub-additive in NIP theories and in simple theories, Chernikov conjectured that it is sub-additive in $\ntptwo$ theories. The main result of \cref{sec:subaddbdn} is the following:

\begin{restatable}{theorem}{subadditivebdn}\label{thm:subadditivebdn}
    Assume $T$ has dependent dividing and $k_1, k_2<\omega$.
    Let $a_1, a_2$ be tuples such that $\bdn(a_i/A))\leq k_i$ for $i\in {1,2}$. Then, $\bdn(a_1a_2/A)\leq k_1+k_2$.
\end{restatable}
Note that for types of infinite burden, sub-multiplicativity already implies sub-additivity.
Therefore, if the dependent dividing conjecture is true, then \cref{thm:subadditivebdn} implies that the burden is sub-additive in $\ntptwo$ theories. The proof of \cref{thm:subadditivebdn} is based on the proof of the sub-additivity of the dp-rank in \cite{KOU11} and the equivalence of the burden and the NIP dp-rank. 
After proving \cref{thm:subadditivebdn}, we state a connection between the NIP dp-rank and the dual VC density in \cref{subsec:vcdim}. This is an immediate consequence of a result in \cite{KOU11}.

\section{Dependent Dividing, NIP dp-rank and Burden}\label{sec:prelim}
In this section, we recall facts about $\ntptwo$ theories and burden. Then, we use them to
prove the equivalence between the burden and the NIP dp-rank assuming that $T$ has dependent dividing. The following are variations of the ict-pattern (which can be used to characterize the dp-rank) and the inp-pattern (which is used to characterize the burden) with the restriction that the witnessing formulas are NIP.
\begin{definition}\label{def:nipcit}
    Assume that $p(x)$ is a (partial) type over $C$. 
    An ict-pattern (ict stands for independent contradictory types) of depth $\kappa$ in $p(x)$ is a sequence $(\phi_\alpha(x,y_\alpha))_{\alpha<\kappa}$ of formulas and an array of tuples $(b_{\alpha,i})_{\alpha<\kappa, i<\omega}$ such that the following set of formulas 
    \[\{\phi_\alpha(x,b_{\alpha,i}):\alpha<\kappa, i<\omega, f(\alpha)=i\}\cup \{\neg\phi_\alpha(x,b_{\alpha,i}):\alpha<\kappa, i<\omega, f(\alpha)\neq i\}\cup p(x)\]
    is consistent for every path $f:\kappa\to \omega$.

    An ict-pattern is called an NIP-ict-pattern if each $\phi_\alpha(x;y_\alpha)$ is NIP.
\end{definition}
\begin{definition}\label{def:nipbdn}
    An inp-pattern consisting of $(b_{\alpha, i})_{\alpha<\kappa, i<\omega}$, $(\phi_\alpha(x,y_\alpha))_{\alpha<\kappa}$ and $k_\alpha<\omega$ is called an NIP-inp-pattern if each $\phi_\alpha(x;y_\alpha)$ is NIP.
\end{definition}
Throughout this paper, given an array $(b_{\alpha, i})_{\alpha<\kappa, i<\omega}$, a sequence $(b_{\alpha, i})_{i<\omega}$ (for fixed $\alpha$) is called a row, and a sequence $(b_{\alpha, i})_{\alpha<\kappa}$ (for fixed $i$) is called a column.

We also remind the readers of strict non-forking and strict invariance:
\begin{definition}\label{def:st}
        We say that $\tp(a/Ab)$ strictly does not fork over $A$ (and write $a\ind_A^{st} b$) if there is a global extension $p$ of $\tp(a/Ab)$ which does not fork over $A$ and for any $B\supseteq Ab$, if $c\vDash p|_B$, then $\tp(B/Ac)$ does not divide over $A$. 
\end{definition}
\begin{definition}\label{def:ist}
        We say that $\tp(a/Ab)$ is strictly invariant over $A$ (and write $a\ind_A^{ist} b$) if there is a global extension $p$ of $\tp(a/Ab)$ which is Lascar invariant over $A$ and for any $B\supseteq Ab$, if $c\vDash p|_B$, then $\tp(B/Ac)$ does not divide over $A$.
        
\end{definition}
\begin{fact}[\cite{KU14}, Definition 3.7]\label{fact:st=ist}
    Generally, $a\ind^{ist}_A b$ implies $a\ind^{st}_A b$.
    If $T$ is NIP, then $a\ind^{st}_A b$ implies $a\ind^{ist}_A b$.
\end{fact}
As an important example, $\tp(a/Mb)$ is strictly invariant over $M$ (where $M$ is a model) if $\tp(a/Mb)$ has a global heir-coheir extension over $M$.

\begin{fact}[\cite{Chernikov14}, Lemma 4.3]\label{fact:enoughsaturation-ist}
    Let $p(x)$ be a global type invariant over $A$, and let $M\supseteq A$ be a $|A|^+$-saturated model. Then, $p(x)$ is an heir over $M$ and hence strictly invariant over $M$.
\end{fact}
Here is a fundamental fact about dividing in $\ntptwo$ theories, which is called ``Kim's lemma in $\ntptwo$ theories." 
\begin{fact}[\cite{CK12}, Lemma 3.14]\label{fact:Kimslemma}
    Assume that $T$ is $\ntptwo$.
    If $\phi(x,a)$ divides over $A$, and $( b_i)_{i<\omega}$ is a sequence satisfying $b_i\equiv_A a$ and $b_i\ind^{ist}_A b_{<i}.$ Then $\{\phi(x,b_i):i<\omega\}$ is inconsistent.
    In particular, if $( b_i)_{i<\omega}$ is an indiscernible sequence, then it witnesses dividing of $\varphi(x,a).$
\end{fact}
We call a sequence $( b_i)_{i<\omega}$ \textit{strictly invariant over $A$} if $b_i\ind^{ist}_A b_{<i}.$ 
Here is another important fact about strictly invariant sequences, which allows us to make indiscernible sequences mutually indiscernible without changing their type over the first element.

\begin{fact}[\cite{KU14}, Lemma 3.11]\label{fact:ShelahLemma}
    Let $( b_\alpha)_{\alpha<\kappa}$ be a strictly invariant sequence over $A$. Then, 
    for every array $(b_{\alpha,i})_{\alpha<\kappa, i<\omega}$ with infinite $A$-indiscernible rows and $b_{\alpha,0}=b_\alpha$, there exists $(b'_{\alpha,i})_{\alpha<\kappa, i<\omega}$ such that 
    \begin{itemize}
        \item $(b_{\alpha,i}')_{i<\omega}\equiv_{b_\alpha A} (b_{\alpha,i})_{i<\omega}$ for each $\alpha<\kappa$.
        \item the rows of $(b'_{\alpha,i})_{\alpha<\kappa, i<\omega}$ are mutually indiscernible over $A.$
    \end{itemize}
\end{fact}
The following fact is a characterization of burden using strict invariance. We include the proof of this fact here because we need it later in the proof of \cref{thm:bdn=nipdp}.
\begin{fact}[\cite{Chernikov14}, Theorem 4.7]\label{fact:charbdn}
    Let $p(x)$ be a type over $C$. The following are equivalent:
    \begin{enumerate}
        \item There is an inp-pattern of depth $\kappa$ in $p(x)$.
        \item There exist $d\vDash p(x)$, $M\supseteq C$, and $(e_\alpha)_{\alpha<\kappa}$ such that $e_\alpha \ind_M^{ist} e_{<\alpha}$ and $d\nind_M^d e_\alpha$ for all $\alpha<\kappa$.
    \end{enumerate}
\end{fact}
\begin{proof}
    (1)$\to$ (2): Assume (1).
    Let $(b_{\alpha, i})_{\alpha<\kappa, i<\omega}$, $(\phi_\alpha(x,y_\alpha))_{\alpha<\kappa}$ and $k_\alpha<\omega$ be an inp-pattern of depth $\kappa$ in $p(x)$ with the rows of $(b_{\alpha, i})_{\alpha<\kappa, i<\omega}$ being mutually indiscernible sequences over $C$.
    For each $\alpha<\kappa$, let $q_\alpha(y_\alpha)$ be a non-algebraic global type finitely satisfiable in  $(b_{\alpha, i})_{ i<\omega}$ extending the type $\tp(b_{\alpha,0}/C).$ Let $M\supseteq C\cup(b_{\alpha, i})_{\alpha<\kappa, i<\omega}$ be a $(|C|+\kappa)^+$-saturated model. 
    Then, by \cref{fact:enoughsaturation-ist}, the type $q_\alpha(y_\alpha)$ is strictly invariant over $M$.
    For $\alpha, i<\kappa$, inductively define 
    \[c_{\alpha,i}\vDash q_\alpha(y_\alpha)|_{M\cup(c_{\alpha, j})_{\alpha<\kappa, j<i}\cup(c_{\beta, i})_{\beta<\alpha}},\]
    and define $e_\alpha:= c_{\alpha, \alpha}.$ In this way, we have the following.
    \begin{itemize}
        \item The constructed sequence $(e_\alpha)_{\alpha<\kappa}$ satisfies $e_\alpha \ind_M^{ist} e_{<\alpha}$ for each $\alpha<\kappa$, since $e_\alpha\vDash q_\alpha(y_\alpha)|_{Me_{<\alpha}}$.
        \item The collection of formulas $\theta(x)= p(x)\cup \{\phi_\alpha(x,e_{\alpha}): \alpha<\kappa\}$ is consistent: For any $\Delta\in \mathcal{L}(C)$ and increasing indices $\alpha_0<\cdots<\alpha_n<\kappa$, if $\vDash \Delta(e_{\alpha_0},\ldots,e_{\alpha_n})$, then $\vDash \Delta(b_{\alpha_0,i_0},\ldots,b_{\alpha_n,i_n})$ for some $i_0,\ldots,i_n<\omega$.
        Since $(b_{\alpha, i})_{\alpha<\kappa, i<\omega}$ is an inp-pattern in $p(x)$, any finite subset of $\theta(x)$ is consistent.
        Thus, by compactness, there is a realization $d\vDash \theta(x)$.
        \item For each $\alpha<\kappa$, the formula $\phi_\alpha(x,e_\alpha)$ divides over $M$: We show that for each $\alpha<\kappa$, the sequence $(c_{\alpha,\alpha+i})_{i<\omega}$ witnesses the dividing. This sequence is $M$-indiscernible because $c_{\alpha,\alpha+i}\vDash q_\alpha(y_\alpha)|_{M(c_{\alpha,\alpha+j})_{j<i}}$ and $q_\alpha$ is finitely satisfiable in $M$.
        As $\tp((c_{\alpha,i})_{i<\omega})$ is finitely satisfiable in $(b_{\alpha, i})_{ i<\omega}$, we conclude that $\{\phi_\alpha(x,c_{\alpha,\alpha+i}):i<\omega\}$ is $k_\alpha$-inconsistent.
    \end{itemize}

    (2)$\to$ (1): Let $d\vDash p(x)$, $M\supseteq C$, and $(e_\alpha)_{\alpha<\kappa}$ such that $e_\alpha \ind_M^{ist} e_{<\alpha}$ and $d\nind_M^d e_\alpha$ for all $\alpha<\kappa$.
    For each $\alpha<\kappa$, let the following be a witness of $d\nind_M^d e_\alpha$: a formula $\phi_\alpha(x,y_\alpha)\in \tp(d/e_\alpha M)$ and a $M$-indiscernible sequence $(e_{\alpha,i})_{i<\omega}$ with $e_{\alpha,0}=e_\alpha$. By \cref{fact:ShelahLemma}, we can find an array $(e'_{\alpha, i})_{\alpha<\kappa, i<\omega}$ with mutually indiscernible rows over $M$ and such that $(e'_{\alpha,i})_{i<\omega}\equiv_{e_\alpha M} (e_{\alpha,i})_{i<\omega}$ for each $\alpha<\kappa$. Thus, $(e'_{\alpha, i})_{\alpha<\kappa, i<\omega}$ and $(\phi_\alpha(x,y_\alpha))_{\alpha<\kappa}$ form an inp-pattern of depth $\kappa$ in $p(x).$
\end{proof}

Note that in the original theorem by Chernikov, the second equivalent condition is stated as: there exists a set $D\supseteq C$ with the desired properties. However, Chernikov's proof in fact shows that there exists a model $M\supseteq C$ with the desired properties so it is cited that way. 

Now, we prove that, under the assumption of dependent dividing, the burden and the NIP-dp rank are equivalent. Note that in the following proof, the only place that uses the assumption of dependent dividing is (1)$\to$(2). 

For a notation, given a tuple $a\in \mathfrak{C}$ and small sets $B,C\subseteq \mathfrak{C}$, we use $\niptp(a/B;C)$ to denote the collection of NIP formulas $\phi(x,b;c)$ with $|x|=|a|$, $b\in B$, $c\in C$ such that $\vDash \phi(a,b;c)$. Note that $\tp(a/B)\subseteq \niptp(a/B;C)$.
We use $\niptp(a/C)$ to denote $\niptp(a/\emptyset;C)$.
\begin{restatable}{theorem}{bdnnipdp}\label{thm:bdn=nipdp}
    Assume that $T$ has dependent dividing and
    $p(x)$ is a (partial) type over $C$. Then, the following are equivalent for any $\kappa$:
    \begin{enumerate}
        \item $\bdn(p(x))\geq \kappa$.
        \item There is an NIP-inp-pattern of depth $\kappa$ in $p(x).$
        \item There is an NIP-ict-pattern of depth $\kappa$ in $p(x).$
        \item $\nipdp(p(x))\geq \kappa$.
        \item There exist $d\vDash p(x)$, some set $D$, and $\{I_\alpha: \alpha<\kappa\}$,
    mutually NIP-indiscernible sequences over $(D;C)$ such that for all $\alpha<\kappa$, $I_\alpha$ is not NIP-indiscernible over $(D;Cd)$.
    \end{enumerate}
\end{restatable}
\begin{proof} 
(1)$\to$(2): This direction is similar to Proposition 5.9 (1) in \cite{Chernikov14}.
Assume (1).
 Let $(b_{\alpha, i})_{\alpha<\kappa, i<\omega}$, $(\phi_\alpha(x,y_\alpha))_{\alpha<\kappa}$ and $k_\alpha<\omega$ be an inp-pattern of depth $\kappa$ in $p(x)$. 
Let $T_{Sk}$ be the Skolemization of $T$.
Note that the inp-pattern of depth $\kappa$ in $p(x)$ in $T$ is still an inp-pattern in $T_{Sk}$.
Now working in $T_{Sk}$, we follow the same construction as in the proof of \cref{fact:charbdn}. We fix $d\vDash p(x)$, $M\supseteq C$, $(e_\alpha)_{\alpha<\kappa}$, and a global type $q_\alpha(y_\alpha)$ for each $\alpha<\kappa$ such that 
\begin{itemize}
    \item the type $q_\alpha(y_\alpha)$ extends $ \tp(b_{\alpha,0}/C)$ and is an heir-coheir over $M$;
    \item by defining $(e_\alpha)_{\alpha<\kappa}$ as a part of consecutive realizations of the global heir-coheir extensions over $M$, we have 
    $e_\alpha \ind_M^{ist} e_{<\alpha}$; 
    \item $d\nind_M^d e_\alpha$ for all $\alpha<\kappa$ witnessed by $\phi_\alpha$.
\end{itemize}

In order to apply dependent dividing, we define models in $T_{Sk}$ by letting 
$M_\alpha:=Sk(M\cup e_\alpha)$ for each $\alpha<\kappa$ (where $Sk(-)$ denotes the Skolem hull).
In this way, we have $M_\alpha \ind_M^{ist} M_{<\alpha}$ because the type $\tp(M_\alpha/M_{<\alpha}M)$ still has a global heir-coheir extension over $M$.

Now, we return to working in $T$.
For each $\alpha<\kappa$, we have
\begin{itemize}
    \item $M_\alpha\vDash T$ by Skolemization;
    \item $M_\alpha \ind_M^{ist} M_{<\alpha}$ because a global heir-coheir extension in $T_{Sk}$ is still a global heir-coheir extension in $T$;
    \item $d\nind_M^d M_\alpha$ as witnessed by $\phi_\alpha$.
\end{itemize}
Because $T$ has dependent dividing, for each $\alpha<\kappa$, we can fix an NIP formula $\psi_\alpha(x;c_\alpha)\in \mathcal{L}(M_\alpha)$ that witnesses $\tp(d/M_\alpha)$ dividing over $M$ with an $M$-indiscernible sequence $(c_{\alpha,i})_{i<\omega}$ (with $c_\alpha=c_{\alpha,0})$. 
By \cref{fact:ShelahLemma}, we can find an array $(c_{\alpha,i}')_{\alpha<\kappa, i<\omega}$ with mutually indiscernible rows over $M$ and with $(c_{\alpha,i}')_{i<\omega}\equiv_{c_\alpha M} (c_{\alpha,i})_{i<\omega}$ for each $\alpha<\kappa$. In this way, 
$(c_{\alpha, i}')_{\alpha<\kappa, i<\omega}$ and $\psi_\alpha(x,z_\alpha)$ (where $|z_\alpha|=|c_\alpha|$) form an
NIP-inp-pattern of depth $\kappa$ in $p(x).$

(2)$\to$(3): Assume that there is an NIP-inp-pattern of depth $\kappa$ in $p(x)$, witnessed by a sequence $(\phi_\alpha(x;y_\alpha))_{\alpha<\kappa}$ of NIP formulas and an array of tuples $(b_{\alpha,i})_{\alpha<\kappa, i<\omega}$. 
By the standard argument, we may assume that the rows of $(b_{\alpha,i})_{\alpha<\kappa, i<\omega}$ are mutually indiscernible over $C$.
By the path consistency, fix $d\vDash p(x)$ with $\vDash \phi(d,b_{\alpha,0})$ for all $\alpha<\kappa$. 
By the row inconsistency, for each $\alpha<\kappa$, there are only finitely many other $0<i<\omega$ with $\vDash \phi(d,b_{\alpha,i})$. 
Therefore, after removing finitely many $b_{\alpha,i}$ for each $\alpha$, we have $\vDash \phi(d,b_{\alpha, i})$ iff $i=0$. By the mutual indiscernibility over $C$, a sequence $(\phi_\alpha(x;y_\alpha))_{\alpha<\kappa}$ of NIP formulas and an array of tuples $(b_{\alpha,i})_{\alpha<\kappa, i<\omega}$ witness an NIP-ict-pattern of depth $\kappa$ in $p(x)$.

(3)$\to$(4): Assume that a sequence $(\phi_\alpha(x;y_\alpha))_{\alpha<\kappa}$ of NIP formulas and an array of tuples $(b_{\alpha,i})_{\alpha<\kappa, i<\omega}$
with $|y_\alpha|=|b_{\alpha,i}|$
forms an NIP-ict-pattern of depth $\kappa$ in $p(x)$. 
We may assume that the rows of $(b_{\alpha,i})_{\alpha<\kappa, i<\omega}$ are mutually NIP-indiscernible over $C$. 
Because each $\phi_\alpha(x;y_\alpha)$ is NIP, we also have $\phi_\alpha^{opp}(y_\alpha;x)$ is NIP. 
Therefore, 
there is $d\vDash p(x)$ so that
each $\phi_\alpha^{opp}(y_\alpha;x)$ witnesses that the sequence $(b_{\alpha,i})_{\alpha<\kappa, i<\omega}$ is not NIP indiscernible over $Cd$.
Thus, we have $\nipdp(p(x))\geq \kappa$.

(4)$\to$(5): immediate. 

(5)$\to$(1): Fix $d\vDash p(x)$, a set $D$, and $\{I_\alpha: \alpha<\kappa\}$ as in (5).
Without loss of generality, assume that each $I_\alpha$ is indexed by $\mathbb{Q}$, so $I_\alpha=(b_{\alpha, q})_{q\in \mathbb{Q}}.$
For each $I_\alpha$, 
because $I_\alpha$ is not NIP-indiscernible over $(D;Cd)$, 
we can find $q_\alpha^1, q_\alpha^2\in \mathbb{Q}$ 
with $q_\alpha^1< q_\alpha^2$ 
and $B_\alpha := \{b_{\alpha, q}: q<q_\alpha^1 \text{ or } q_\alpha^2< q \}$
such that 
$\niptp(b_{\alpha, q_{\alpha}^1}/B_\alpha D;C d)\neq \niptp(b_{\alpha, q_{\alpha}^2}/B_{\alpha}D;Cd)$; this is witnessed by some formula $\phi_\alpha$ and some finite subsets $\overline{B_\alpha}\subseteq B_\alpha$, $\overline{D}\subseteq {D}$, 
$\overline{C}\subseteq  {C}$ such that
\[\vDash \phi_\alpha(b_{\alpha, q_{\alpha}^1},\overline{B_\alpha}\overline{D};\overline{C} d)\land \neg \phi_\alpha(b_{\alpha, q_{\alpha}^2},\overline{B_{\alpha}D};\overline{C}d)\]
so that the partitioned formula $\phi_\alpha(z; w)$, where $|z|=|b_{\alpha, q_{\alpha}^1},\overline{B_\alpha}\overline{D}|$ and $|w|=|\overline{C} d|$, is NIP.
Let
$I_\alpha' = (b_{\alpha, q}\overline{B_\alpha DC})_{q\in \mathbb{Q}, q_\alpha^1\leq q\leq q_\alpha^2}$.

Without loss of generality, we can let $I_\alpha' = (b_{\alpha, i}\overline{B_\alpha DC})_{i\in \omega}$ with $\vDash \phi_\alpha(b_{\alpha, 0}\overline{B_\alpha DC},d)\land \neg \phi_\alpha(b_{\alpha, 1}\overline{B_\alpha DC},d)$. 
For each $\alpha<\kappa$, let
$\psi_\alpha(y, x_1,x_2):= \phi_\alpha(x_1, y)\land \neg \phi_\alpha(x_2, y)$ (where $|y|=|d|$) and define $J_\alpha := (b_{\alpha, 2i}\overline{B_\alpha DC}, b_{\alpha, 2i+1}\overline{B_\alpha DC})_{i<\omega}$.
Now, we show that $\psi_\alpha(y,x_1,x_2)$ and $(J_\alpha)_{\alpha<\kappa}$ form an inp-pattern of depth $\kappa$.
Because each $\phi_\alpha(b_{\alpha, 0}\overline{B_\alpha D};\overline{C},d)$ is NIP, we know that it has finite alternation rank. Since NIP-indiscernibility over $(D;C)$ implies indiscernibility over $D$, $(b_{\alpha,i}\overline{B_\alpha D})_{i<\omega} $ is an indiscernible sequence, and thus
\[\{\psi_\alpha(y, b_{\alpha, 2i}\overline{B_\alpha DC}, b_{\alpha, 2i+1}\overline{B_\alpha DC}): i<\omega\}\]
is $k_\alpha$-inconsistent for every $\alpha$.
For any path $f: \kappa\to \omega$, consider the set  
\[\pi(y)=
    p(y)\cup \{\psi_\alpha(y, b_{\alpha, 2f(\alpha)}\overline{B_\alpha DC}, b_{\alpha, 2f(\alpha)+1}\overline{B_\alpha DC}): \alpha<\kappa\}.
\]
Fix a finite subset $\pi_0(y)\subseteq \pi(y)$, then 
\begin{multline*}
    \pi_0(y)=\theta(y,\overline{C})\land \psi_{\alpha_0}(y, b_{{\alpha_0}, 2f({\alpha_0})}\overline{B_{\alpha_0} DC}, b_{{\alpha_0}, 2f({\alpha_0})+1}\overline{B_{\alpha_0} DC})\\\land\cdots\land \psi_n(y, b_{{\alpha_n}, 2f({\alpha_n})}\overline{B_{\alpha_n} DC}, b_{{\alpha_n}, 2f({\alpha_n})+1}\overline{B_{\alpha_n} DC})
\end{multline*}
where $\theta(y,\overline{C})\in p(y)$ with $\overline{C}\subseteq C$ and ${\alpha_0},...,{\alpha_n}<\kappa$. 
Note that the formula
\begin{multline*}
    \exists y\theta(y,\overline{C})\land \phi_{\alpha_0}(b_{{\alpha_0}, 0}\overline{B_{\alpha_0} D} ;\overline{C}y)\land \neg \phi_{\alpha_0}(b_{{\alpha_0}, 1}\overline{B_{\alpha_0} D}; \overline{C}y)\\\land \cdots \land\phi_{\alpha_n}(b_{{\alpha_n}, 0}\overline{B_{\alpha_n} D} ;\overline{C}y)\land \neg \phi_{\alpha_n}(b_{{\alpha_n}, 1}\overline{B_{\alpha_n} D}; \overline{C}y)
\end{multline*}
is in $\niptp(b_{{\alpha_0}, 0}b_{{\alpha_0},1}\overline{B_{\alpha_0}}/D(I_\beta')_{\beta\neq {\alpha_0}};C)$ (if we replace the instances of $b_{{\alpha_0}, 0}b_{{\alpha_0},1}\overline{B_{\alpha_0}}$ with free variables) because \[\phi_{\alpha_0}(b_{{\alpha_0}, 0}\overline{B_{\alpha_0} D} ;\overline{C}y)\land \neg \phi_{\alpha_0}(b_{{\alpha_0}, 1}\overline{B_{\alpha_0} D}; \overline{C}y)\\\land \cdots \land\phi_{\alpha_n}(b_{{\alpha_n}, 0}\overline{B_{\alpha_n} D} ;\overline{C}y)\land \neg \phi_{\alpha_n}(b_{{\alpha_n}, 1}\overline{B_{\alpha_n} D}; \overline{C}y)\] is NIP.
Because $\{I_\alpha: \alpha<\kappa\}$ is mutually NIP indiscernible over $(D;C)$, $I_{\alpha_0}$ is NIP-indiscernible over $(D(I_\beta')_{\beta\neq {\alpha_0}}; C)$, so
\begin{multline*}
    \vDash \exists y\theta(y,\overline{C})\land \phi_{\alpha_0}(b_{{\alpha_0}, 2f(\alpha_0)}\overline{B_{\alpha_0}D} ;\overline{C}y)\land \neg \phi_{\alpha_0}(b_{{\alpha_0}, 2f(\alpha_0)+1}\overline{B_{\alpha_0}D}; \overline{C}y)
    \\\land \cdots \land\phi_{\alpha_n}(b_{{\alpha_n}, 0}\overline{B_{\alpha_n}D} ;\overline{C}y)\land \neg \phi_{\alpha_n}(b_{{\alpha_n}, 1}\overline{B_{\alpha_n}D}; \overline{C}y)
\end{multline*}
as well. If we repeat this process for the rest of the rows ${\alpha_1},...,{\alpha_n}<\kappa$, then we can conclude that $\pi_0(y)$ is consistent. 
\end{proof}

\section{Dependent Dividing and Sub-additivity of Burden}\label{sec:subaddbdn}
In this section, we show that dependent dividing implies sub-additivity of burden, following the proof of the sub-additivity of dp-rank in \cite{KOU11}.
We start by stating known facts.
\begin{fact}
    In NIP theories, the burden agrees with the dp-rank \cite[Proposition 10]{Adler07}.
    Because the dp-rank is sub-additive by \cite{KOU11}, the burden is sub-additive in NIP theories.
\end{fact}

\begin{fact} 
    In simple theories, the burden of a partial type is the supremum of the weights of its complete extensions. 
    Because the weight is sub-additive by Chapter XIII, Theorem 2.12 in \cite{Baldwin17}, the burden is sub-additive in simple theories.
\end{fact}

\begin{fact}[\cite{Chernikov14}, Corollary 2.6]
    In any theory, the burden is sub-multiplicative.
\end{fact}
Based on these theorems, Chernikov conjectured that the burden is sub-additive in $\ntptwo$ theories \cite[Conjecture 2.7]{Chernikov14}.

\subsection{Sub-additivity of NIP dp-rank: NIP dp-minimal}
We start with the NIP dp-minimal case. The theorem follows from \cref{thm:bdn=nipdp} and replacing ``indiscernible" in the proof in \cite{KOU11} with ``NIP-indiscernible" appropriately. We start by recalling the definition of an average type. 
\begin{definition}\label{def:avetype}
    Let $I=\langle a_i\rangle$ be an indiscernible sequence and $\mathcal{U}$ an ultrafilter on the index set of $I$. 
    Given any set $B$, we define $Avg_\mathcal{U}(I,B)$, the average type of $I$ over $B$ given by $\mathcal{U}$, as the unique complete type $p(x)$ such that for every formula $\varphi(x,y)$ and $b\in B$, we have
    \[\varphi(x,b)\in p(x)\iff \{i:\vDash \varphi(a_i,b)\}\in \mathcal{U}.\]
\end{definition}

\begin{proposition}\label{prop:extendindisc}
    Let $I$ be an infinite indiscernible sequence, indexed by an order with no last element, and let $B,C$ be any set. Then, for any indexing set $\lambda$, there is an indiscernible sequence $I^\ast$ indexed by $\lambda$ so that
$I^\frown I^\ast$ is NIP-indiscernible over $(B';C')$ whenever $B'\subseteq B$ and $C'\subseteq C$ are sets
such that $I$ was already NIP-indiscernible over $(B';C')$.
\end{proposition}
\begin{proof}
    Let $I=\langle a_i:i\in \mathcal{I}\rangle$.
    Let $\mathcal{U}$ be an ultrafilter over $I$ such that every set in $\mathcal{U}$ is unbounded in $I$.
    We inductively define 
    $a_0^\ast:= Avg_\mathcal{U}(I, BCI)$ and 
    \[a_{n+1}^\ast:= Avg_\mathcal{U}(I, BCI^\frown\langle a_n^\ast,a_{n-1}^\ast,\ldots,a_0^\ast\rangle).\]
    First, we show that $I^\frown\langle a_0^\ast \rangle$ is NIP-indiscernible over $(B';C')$ whenever $I$ was NIP-indiscernible over $(B';C').$
    Let $\phi(x_0,\ldots,x_n,y;z)$ be an NIP formula with $|x_0|=\cdots=|x_n|=|a_i|$.
    Then for any tuple $b'\in B'$ of length ${|y|}$, tuple $c'\in C'$ of length $|z|$, and $i_0<\ldots<i_n$ and $ j_0<\ldots<j_{n}$ from $ \mathcal{I}$, we have \[\vDash \phi(a_{i_0},\ldots a_{i_{n}},b'; c')\leftrightarrow \phi(a_{j_0},\ldots, a_{j_n},b'; c').\]
    Then, for any $i_0<\ldots<i_{n-1}$ and $ j_0<\ldots<j_{n-1}$ from $ \mathcal{I}$, we have 
    \[\vDash \phi(a_{i_0},\ldots a_{i_{n-1}},a_0^\ast,b'; c')\leftrightarrow \phi(a_{j_0},\ldots, a_{j_{n-1}}, a_0^\ast,b'; c')\]
    because
    $ \phi(a_{i_0},\ldots a_{i_{n-1}},x,b'; c')\leftrightarrow \phi(a_{j_0},\ldots, a_{j_{n-1}}, x,b'; c')$
    is satisfied by unboundedly many elements in $I$.

    Inductively, assume $I^\frown \langle a_n^\ast,a_{n-1}^\ast,\ldots,a_0^\ast\rangle$ is NIP-indiscernible over $(B';C')$ whenever $I$ was NIP-indiscernible over $(B';C')$.
    Then, by a very similar argument, it follows that $I^\frown \langle a_{n+1}^\ast, a_n^\ast,a_{n-1}^\ast,\ldots,a_0^\ast\rangle$ is NIP-indiscernible over $(B';C')$ whenever $I$ was NIP-indiscernible over $(B';C')$.
    
    This construction gives us an arbitrarily long finite continuation $I^\ast$ of the original sequence $I$ with the desired property. By compactness, for any indexing set $\lambda$, a continuation $I^\ast$ with the desired property exists. 
\end{proof}

\begin{corollary}\label{cor:extendindisc}
    Let $I$ and $J$ be infinite mutually NIP-indiscernible sequences over $(B';C')$.
    Let $I^\ast$ be as in \cref{prop:extendindisc} with some sets $B\supseteq JB'$ and $C\supseteq C'$. 
    Then, $I^\frown I^\ast$ and $J$ are mutually NIP-indiscernible over $(B';C')$.
\end{corollary}
\begin{proof}
    By \cref{prop:extendindisc}, $I^\frown I^\ast$ is NIP-indiscernible over $(JB';C')$.
    To show that $J$ is NIP-indiscernible over $(I^\frown I^\ast B';C')$, assume, towards contradiction, that $J$ is not NIP-indiscernible over $(I^\frown I^\ast B';C')$.
    Fix a finite tuple $\overline{a}^\frown \overline{b}$ with $\overline{a}\in I$ and $\overline{b}\in I^\ast$ so that $J$ is not NIP-indiscernible over $(\overline{a}\overline{b} B';C')$, witnessed by finite subsets $\overline{B'}\subseteq B'$ and $\overline{C'}\subseteq C'$ with
    \[\vDash \phi(j_{i_0},\ldots,j_{i_n},\overline{a}, \overline{b},\overline{B'}; \overline{C'})\land \neg\phi(j_{k_0},\ldots,j_{k_n},\overline{a}, \overline{b}, \overline{B'}; \overline{C'}).\]
    Since $I^\frown I^\ast$ is NIP-indiscernible over $(JB';C')$ and $I$ is infinite, there are some $\overline{a}', \overline{b}'\in I$ such that 
    \[\vDash \phi(j_{i_0},\ldots,j_{i_n},\overline{a}', \overline{b}',\overline{B'}; \overline{C'})\land \neg\phi(j_{k_0},\ldots,j_{k_n},\overline{a}', \overline{b}',\overline{B'}; \overline{C'}).\]
    But this would imply that $J$ is not NIP-indiscernible over $(IB';C')$, a contradiction. 
\end{proof}

\begin{lemma}\label{lem:minlem}
Assume that $T$ has dependent dividing. 
Let $a$ be a tuple such that $\bdn(a/A)=1$, let $B$ be some set, and let $\mathcal{I}$ be a set of mutually $(B;A)$-NIP-indiscernible sequences. 
Then, for any $n$, given any $n+1$ mutually $(B;A)$-NIP-indiscernible sequences in $\mathcal{I}$, at least $n$ of them are mutually NIP-indiscernible over $(B;Aa).$
\end{lemma}
\begin{proof}
    We prove the statement by induction on $n$. For $n=1$, we need to show that given any two mutually $(B;A)$-NIP-indiscernible sequences in $\mathcal{I}$, at least one of them is NIP-indiscernible over $(B;Aa)$. This is true by (1)$\leftrightarrow$(5) in \cref{thm:bdn=nipdp}.

    For the inductive hypothesis, assume that given any $n$ mutually $(B;A)$-NIP-indiscernible sequences in $I$, at least $n-1$ of them are mutually NIP-indiscernible over $(B;Aa).$
    Now, let $I_1,\ldots, I_{n+1}$ be $n+1$ mutually $(B;A)$-NIP-indiscernible sequences.
    By the definition of mutual indiscernibility, we know that 
    $I_1,\ldots, I_{n}$ are mutually indiscernible over $(BI_{n+1};A)$. 
    Therefore, by the inductive hypothesis, at least $n-1$ of them are mutually NIP-indiscernible over $(BI_{n+1}; Aa)$.
    Without loss of generality, assume that $I_1,\ldots, I_{n-1}$ are mutually NIP-indiscernible over $(BI_{n+1}; Aa)$.
    
    If $I_{n+1}$ was NIP-indiscernible over $(BI_1\ldots I_{n-1}; Aa)$, then $I_1,\ldots, I_{n-1}, I_{n+1}$ would be mutually NIP-indiscernible over $(B;Aa)$. If this is the case, then we would be done, so assume that this is not the case.

    Since non-NIP-indiscernibility can be witnessed by a finite sequence, we assume for the rest of the proof that $I_{n+1}$ is not NIP-indiscernible over $(B\overline{b}; Aa)$ for some $\overline{b}\subseteq \cup\{I_1,\ldots, I_{n-1}\}$.
    Consider the following cases:
    \begin{enumerate}
        \item $I_{n+1}$ is not NIP-indiscernible over $(B; Aa)$.
        
        Since $I_2,I_3,\ldots,I_n,I_{n+1}$ are mutually NIP-indiscernible over $(BI_1; A)$, by the induction hypothesis, at least $n-1$ of them are mutually-NIP indiscernible over $(BI_1; Aa)$. Since $I_{n+1}$ is not NIP-indiscernible over $(B;Aa)$, we know $I_2,I_3,\ldots,I_n$ are mutually NIP-indiscernible over $(BI_1;Aa)$. 
    With the same argument, we can show that $I_1,I_3,\ldots,I_n$ are mutually NIP-indiscernible over $(BI_2; Aa)$. 
    Therefore, $I_1,I_2,\ldots, I_{n}$ are mutually NIP-indiscernible over $(B;Aa)$.
        \item $I_{n+1}$ is not NIP-indiscernible over $(B\overline{b};Aa)$ for some $\overline{b}\subseteq \cup\{I_1,\ldots, I_{n-1}\}$.

        We argue that we can reduce this case to case (1). For each $k$ with $1\leq k<n$, we inductively define a continuation $I^\ast_k$ of $I_k$. 
        After picking $I_j^\ast$ for $j<k$, define $I_k^\ast$ be a sequence indexed by $\omega$ as in \cref{prop:extendindisc} so that whenever 
        $B'\subseteq B\cup \cup_{i=1}^n I_i\cup \cup_{j=1}^{k-1} I_j^\ast$ and $A'\subseteq Aa$ are such that $I_k$ is NIP-indiscernible over $(B';A')$, the sequence $I_k^\frown I_k^\ast$ is NIP-indiscernible over $(B';A')$. 
        By \cref{cor:extendindisc}, we have 
        \begin{itemize}
            \item $I_1^\frown I_1^\ast, \ldots, I_{n-1}^\frown I_{n-1}^\ast, I_n, I_{n+1}$ are mutually $(B;A)$-NIP-indiscernible,
            \item $I_1^\frown I_1^\ast, \ldots, I_{n-1}^\frown I_{n-1}^\ast$ are mutually 
            NIP-indiscernible over $(BI_{n+1};Aa)$, and
            \item $I_{n+1}$ is not NIP-indiscernible over $(B\overline{b};Aa)$ for some $\overline{b}\subseteq \cup\{I_1,\ldots, I_{n-1}\}$.
        \end{itemize}
        Since $I_1^\frown I_1^\ast, \ldots, I_{n-1}^\frown I_{n-1}^\ast$ are mutually 
            NIP-indiscernible over $(BI_{n+1};Aa)$, we can find 
            $\overline{b}'\subseteq \cup_{i=1}^{n-1} I^\ast_i$
            with $\overline{b}'\vDash \niptp(\overline{b}/BI_{n+1};Aa)$ such that
        \begin{itemize}
            \item $I_1, \ldots, I_{n-1}, I_n, I_{n+1}$ are mutually $(B\overline{b}';A)$-NIP-indiscernible,
            \item $I_1, \ldots, I_{n-1}$ are mutually NIP-indiscernible over $(B\overline{b}'I_{n+1};Aa)$, and
            \item $I_{n+1}$ is not NIP-indiscernible over $(B\overline{b}';Aa)$.
        \end{itemize}
        In this way, we can apply case (1) with $B=B\overline{b}'$.
    \end{enumerate}
\end{proof}

\begin{corollary}\label{cor:minsubadditive}
Assume that $T$ has dependent dividing.
Let $\tp(a_i/A)$ be such that $\bdn(a_i/A)=1$ for $1\leq i\leq k$.
Then, $\bdn(a_1\ldots a_k/A)\leq k$.
\end{corollary}
\begin{proof}
We prove the statement by induction on $k$. When $k=1$, the statement is trivial. 

For the inductive hypothesis, assume that the corollary holds for tuples of fewer than $k$ elements. To show that $\bdn(a_1\ldots a_k/A)\leq k$, let $I_1,\ldots, I_{k+1}$ be mutually NIP-indiscernible sequences over $A$. 
By \cref{lem:minlem}, at least $k$ of them are  mutually NIP-indiscernible over $Aa_{k}$.
Without loss of generality, assume that $I_1,\ldots,I_{k}$ are mutually NIP-indiscernible over $Aa_{k}$.
Notice that by the inductive hypothesis and the definition of burden, we know that \[\bdn(a_1\ldots a_{k-1}/Aa_{k})\leq \bdn(a_1\ldots a_{k-1}/A)\leq k-1.\] 
By \cref{thm:bdn=nipdp}, at least one of $I_1,\ldots,I_{k}$ is NIP-indiscernible over $Aa_{1}\ldots a_k$.
\end{proof}

\subsection{Sub-additivity of NIP dp-rank: NIP dp-finite}
The following proposition is the analogue of \cref{lem:minlem} in the finite NIP-dp-rank setting.
\begin{proposition}\label{prop:finprop}
    Assume that $T$ has dependent dividing.
    Let $a$ be an element such that $\bdn(a/A) \leq k$, $B$ be some set, and let $\mathcal{I}:=\{I_1,\ldots, I_m\}$ be a set of mutually $(B;A)$-indiscernible sequences with $m>k$.
    Then, there is a subset of $\mathcal{I}$ of size $m-k$ that is mutually NIP-indiscernible over $(B;Aa)$.
\end{proposition}
The main theorem, \cref{thm:subadditivebdn}, follows from \cref{prop:finprop}. We show \cref{prop:finprop} by breaking it up into multiple smaller lemmas. 
\begin{definition}\label{def:fin}
    Let $\mathcal{I}:=\{I_1,\ldots, I_m\}$ be a set of mutually NIP-indiscernible sequences over $A$, and let $a$ be a tuple. 
    We say that the pair $\mathcal{I}, a$ satisfies $S_{k,n}$ if the following conditions hold:
    \begin{itemize}
        \item $\mathcal{I}\geq k+n$,
        \item For any set $B$ such that $\mathcal{I}:=\{I_1,\ldots, I_m\}$ is a set of mutually $(B;A)$-NIP-indiscernible sequences, given any $n+k$ sequences in $\mathcal{I}$ at least $n$ of them are mutually NIP indiscernible over $(B;Aa)$.
    \end{itemize}
\end{definition}
In particular, in a theory with dependent dividing, a type $p(x)$ over $A$ has $\bdn(p(x))\leq k$ iff for any realization $a$ of $p(x)$ and every set $\mathcal{I}$ of mutually NIP-indiscernible sequences over $A$ where $|\mathcal{I}|\geq k+1$, we have that $\mathcal{I}, a$ satisfies $S_{k,1}$.

\begin{lemma}\label{lem:fin}
    Let $a$ be an element, and let $\mathcal{I}$ be a set of mutually NIP-indiscernibile sequences over $(B;A)$. 
    Let $\mathcal{J}\subseteq \mathcal{I}$. Let $I\in \mathcal{I}$ be such that $\mathcal{J}$ is mutually NIP indiscernible over $(BI;Aa)$ and such that $I$ is not NIP-indiscernible over $(B\mathcal{J};Aa)$.
    Then, there is $B'\supseteq B$ such that the following holds:
    \begin{itemize}
        \item $\mathcal{I}$ is mutually NIP-indiscernible over $(B';A).$
        \item $I$ is not NIP-indiscernible over $(B';Aa)$.
        \item $\mathcal{J}$ is mutually NIP-indiscernible over $(B'I;Aa)$.
    \end{itemize}
\end{lemma}
\begin{proof}
  We show this statement for finite $\mathcal{J}$, which is what we need to prove \cref{thm:subadditivebdn}. The proof for infinite $\mathcal{J}$ goes similarly to the finite case, using transfinite induction. 

  Let $\mathcal{J}=\{J_1,\ldots, J_n\}$ and using \cref{prop:extendindisc}, define a continuation $J_t^\ast := \langle a_i^\ast \rangle_{i\in \omega} $ for every sequence $J_t\in \mathcal{J}$ inductively (on $t$) having that $J_t^\frown J_t^\ast$ is NIP-indiscernible over any $(B'';A')$ with
  \[B''\subseteq B \cup \mathcal{I}\cup \cup_{j=1}^{t-1} J_j^\ast, A'\subseteq Aa\]
  over which $J_t$ was already NIP-indiscernible.

  Because $\mathcal{J}$ is mutually NIP-indiscernible over $(BI;Aa)$, it follows from \cref{cor:extendindisc} that 
  \begin{itemize}
      \item $\{J_1^\frown J_1^\ast,\ldots, J_n^\frown J_n^\ast\}\cup (\mathcal{I}\backslash \mathcal{J})$ is mutually NIP-indiscernible over $(B;A)$,
      \item $\{J_1^\frown J_1^\ast,\ldots, J_n^\frown J_n^\ast\}$ is mutually NIP-indiscernible over $(BI;Aa)$, and
      \item $I$ is not NIP-indiscernible over $(B\overline{b}; Aa)$ for some $\overline{b}\in \cup \{J_1,\ldots, J_n\}.$
  \end{itemize}
  Since $\{J_1^\frown J_1^\ast,\ldots, J_n^\frown J_n^\ast\}$ is NIP-indiscernible over $(BI;Aa)$, we can fix $\overline{b}'\in \cup \{J_1^\ast,\ldots, J_n^\ast\} $ such that $ \overline{b}'\vDash \niptp(\overline{b}/BI;Aa)$.

  Now, we have 
  \begin{itemize}
      \item $\mathcal{I}$ is mutually indiscernible over $(B\overline{b}';A)$,
      \item $\{J_1,\ldots, J_n\}$ is mutually NIP-indiscernible over $(IB\overline{b}';Aa)$, and
      \item $I$ is not NIP-indiscernible over $(B\overline{b}';Aa)$.
  \end{itemize}
  By letting $B':= B\overline{b}'$, we have proved the claim. 
\end{proof}
Here is a generalization of \cref{prop:finprop} with the new notation. 
\begin{proposition}
    Let $a$ be an element, $n$ be any natural number, and let $\mathcal{I}:=\{I_1,\ldots, I_m\}$ be mutually $(B;A)$-NIP-indiscernible sequences with $m\geq k+n$ such that $\mathcal{I},a$ satisfies $S_{k,1}.$
    Then, $\mathcal{I}, a$ satisfies $S_{k,n}.$
\end{proposition}
\begin{proof}
    Let $k$ be arbitrary.
    We show that $S_{k,n}$ implies $S_{k,n+1}$ for all $n$, by induction on $n$.
    Let $a$ be a tuple and $\mathcal{I}:=\{I_1,\ldots, I_m\}$ be mutually $(B;A)$-NIP-indiscernible with $m\geq k+n+1$ such that $\mathcal{I},a$ satisfies $S_{k,i}$ for all $1\leq i\leq n$ (though we only use $S_{k,1}$ and $S_{k,n}$). 
    Let $\mathcal{I}':=\{I_1,\ldots, I_{k+n+1}\}$ be a subset of $\mathcal{I}.$
    We prove that $\mathcal{I}'$ contains a subset of size $n+1$ of sequences which are mutually NIP-indiscernible over $(B; Aa)$.

Let $I_i$ be any sequence in $\mathcal{I}'$. Since $\mathcal{I}'\backslash\{I_i\}$ is a set of $n+k$ mutually NIP-indiscernible sequence over $(BI_i;A)$, there is a subset $\mathcal{I}_i$ of size $n$ which are mutually NIP-indiscernible over $(BI_i;Aa)$.
    If $I_i$ is NIP-indiscernible over $(B\mathcal{I}_i;Aa)$, then this yields a set of size $n+1$ of mutually NIP-indiscernible sequences over $(B;Aa)$, proving that $\mathcal{I}, a$ satisfies $S_{k, n+1}.$
    Towards contradiction, we assume that for every $i$, the sequence $I_i$ is not NIP-indiscernible over $(B\mathcal{I}_i;Aa)$.

    Now, we can apply \cref{lem:fin} to each $I_i$, to find a set $B'\supseteq B$ such that $\mathcal{I}'$ is mutually NIP-indiscernible over $(B';A)$, $I_i$ is not NIP-indiscernible over $(B';Aa)$, and $\mathcal{I}_i$ is mutually NIP-indiscernible over $(B'I_i;Aa)$. 
    If we repeat this for every $i$, then we obtain $B''\supseteq B$ so that each $I_i$ is not NIP-indiscernible over $(B'';Aa)$ and $\mathcal{I}'$ is mutually NIP-indiscernible over $(B'';A)$. This contradicts $S_{k,1}$ of $\mathcal{I}.$
\end{proof}
This concludes the proof of \cref{prop:finprop}. Finally, here is the proof of the main theorem.
\subadditivebdn*
\begin{proof}
    Let $\mathcal{I}=\{I_1,\ldots, I_{k_1+k_2+1}\}$ be mutually NIP-indiscernible sequences over $A$. 
    By \cref{prop:finprop} applied to $a_1$ and $\mathcal{I}$, we can find a subset $\mathcal{I}_1\subseteq \mathcal{I}$ of size $k_2+1$ so that $\mathcal{I}_1$ is a collection of mutually NIP-indiscernible sequences over $Aa_1$. 
    Because $\bdn(a_2/Aa_1)\leq \bdn(a_2/A)\leq k_2$, by \cref{thm:bdn=nipdp}, there exists a sequence $I'\in \mathcal{I}_1$ that is NIP-indiscernible over $Aa_1a_2.$ 
\end{proof}

\subsection{Connection between the Burden and $VC^\ast$-dimension}\label{subsec:vcdim}
Finally, we note that by using \cref{thm:bdn=nipdp}, we can translate the connection between the dp-rank and the dual VC density to the connection between the burden and the dual VC density (assuming that $T$ has dependent dividing).

Let $S_\Delta^{p(y)}(A)$ denote the set of all $\Delta$-types over $A$ consistent with $p(y).$
\begin{definition}
    The $VC^\ast_\Delta$-density of a type $p(y)$ over a set $C$ is 
    \[\inf\{r\in \rr^{\geq 0}: |S_\Delta^{p(y)}(A)|=O(|A|^r) \text{ for all finite $A\subseteq C^{|y|}$}\}.\]
\end{definition}
Formally, this means that there exists a function $f:\nn \to \rr_+$ such that $f=O(n^r)$, and $|S_\Delta^{p(y)}(A)|\leq f(|A|)$ for all $A\subseteq C$ finite. 

\begin{fact}[\cite{KOU11}, Proposition 5.2]\label{dp-vc}
    Let $p(y)$ be a type over $C$ and $\Delta$ be a set of formulas which is closed under boolean combinations. Then, the following are equivalent for $k<\omega$.
    \begin{enumerate}
        \item There is an ict-pattern of depth $k$ witnessed by witnessed by formulas in $\Delta$. 
        \item  There is an $C$-indiscernible sequence $I$ and some formula $\varphi(x,y)\in \Delta$ such that $p(x)$ has $VC^\ast_\varphi$-density of at least $k$ over $I$.
        \item There is an $C$-indiscernible sequence $I$ and some formula $\varphi(x,y)\in \Delta$ such that $p(x)$ has $VC^\ast$-density bigger than $k-1$ with respect to $\varphi(x,y)$ over $I$.
    \end{enumerate}
\end{fact}
We obtain the following corollary by letting $\Delta$ be the set of NIP formulas (which is closed under boolean combinations), and using \cref{thm:bdn=nipdp}.
\begin{corollary}\label{bdn-vc}
    Assume that $T$ has dependent dividing. Let $p(y)$ be a type over $C$. Then, the following are equivalent.
        \begin{enumerate}
        \item $\bdn(p)\geq k$.
        \item  There is an $C$-indiscernible sequence $I$ and an NIP formula $\varphi(x,y)$ such that $p(x)$ has $VC^\ast_\varphi$-density of at least $k$ over $I$.
        \item There is an $C$-indiscernible sequence $I$ and an NIP formula $\varphi(x,y)$ such that $p(x)$ has $VC^\ast$-density bigger than $k-1$ with respect to $\varphi(x,y)$ over $I$.
    \end{enumerate}
\end{corollary}

\subsection*{Acknowledgments} The results of this paper are part of the author’s Ph.D. thesis, supervised by Tom Scanlon. The author thanks him for his guidance and support. 
The author also thanks Nick Ramsey for his lecture videos on model-theoretic tree properties, which sparked the author's interest in $\ntptwo$ theories and burden.

This research was partially supported by the ANRI Fellowship and a scholarship from the Japan Student Services Organization (JASSO). 

\bibliography{references}

@article{CK12,
 ISSN = {00224812},
 URL = {http://www.jstor.org/stable/23208231},
 author = {Artem Chernikov and Itay Kaplan},
 journal = {The Journal of Symbolic Logic},
 number = {1},
 pages = {1--20},
 publisher = {Association for Symbolic Logic},
 title = {FORKING AND DIVIDING IN NTP2 THEORIES},
 urldate = {2024-08-17},
 volume = {77},
 year = {2012}}

@article{KOU11,
author = {Kaplan, Itay and Onshuus, Alf and Usvyatsov, Alexander},
year = {2011},
pages = {},
title = {Additivity of the Dp-rank},
volume = {365},
journal = {Transactions of the American Mathematical Society},
doi = {10.1090/S0002-9947-2013-05782-0}
}

@article{Chernikov14,
title = {Theories without the tree property of the second kind},
journal = {Annals of Pure and Applied Logic},
volume = {165},
number = {2},
pages = {695-723},
year = {2014},
issn = {0168-0072},
doi = {https://doi.org/10.1016/j.apal.2013.10.002},
url = {https://www.sciencedirect.com/science/article/pii/S0168007213001504},
author = {Artem Chernikov}
}

@book{Simon15,
author = {Simon, Pierre},
year = {2015},
month = {07},
pages = {},
title = {A Guide to NIP Theories},
isbn = {9781107057753},
doi = {10.1017/CBO9781107415133}
}

@article{OS11,
 ISSN = {00224812},
 URL = {http://www.jstor.org/stable/23041846},
 author = {Alf Onshuus and Alexander Usvyatsov},
 journal = {The Journal of Symbolic Logic},
 number = {3},
 pages = {737--758},
 publisher = {[Association for Symbolic Logic, Cambridge University Press]},
 title = {ON DP-MINIMALITY, STRONG DEPENDENCE AND WEIGHT},
 urldate = {2024-09-15},
 volume = {76},
 year = {2011}
}

@unpublished{Adler07,
  author = "Hans Adler",
  title  = "Strong Theories, Burden, and Weight",
  year = "2007",
  note = "Unpublished note, https://citeseerx.ist.psu.edu/document?repid=rep1\&type=pdf\&doi=b7fc9de001c1eea5f29f5c2afb49befa8439a556"
}

@phdthesis{Johnson16,
author = {Johnson, Will},
year = {2016},
month = {05},
pages = {},
title = {Fun with Fields}
}

@article{DG17,
author = {Dolich, Alfred and Goodrick, John},
year = {2017},
pages = {269-296},
title = {Strong theories of ordered Abelian groups},
volume = {236},
journal = {Fundamenta Mathematicae},
doi = {10.4064/fm256-5-2016}
}

@article{DG25, title={DISCRETE SETS DEFINABLE IN STRONG EXPANSIONS OF ORDERED ABELIAN GROUPS}, volume={90}, DOI={10.1017/jsl.2024.43}, number={1}, journal={The Journal of Symbolic Logic}, author={Dolich, Alfred and Goodrick, John}, year={2025}, pages={423–459}}

@article{DG23,
author = {Dolich, Alfred and Goodrick, John},
title = {Topological properties of definable sets in ordered Abelian groups of burden 2},
journal = {Mathematical Logic Quarterly},
volume = {69},
number = {2},
pages = {147-164},
doi = {https://doi.org/10.1002/malq.202200052},
year = {2023}
}

@article{KU14,
author = {Kaplan, Itay and Usvyatsov, Alexander},
title = {Strict independence},
journal = {Journal of Mathematical Logic},
volume = {14},
number = {02},
pages = {1450008},
year = {2014},
doi = {10.1142/S0219061314500081},
URL = {https://doi.org/10.1142/S0219061314500081},}

@article{KS14,
author = {Itay Kaplan and Pierre Simon},
title = {{Witnessing Dp-Rank}},
volume = {55},
journal = {Notre Dame Journal of Formal Logic},
number = {3},
publisher = {Duke University Press},
pages = {419 -- 429},
year = {2014},
doi = {10.1215/00294527-2688105},
URL = {https://doi.org/10.1215/00294527-2688105}
}

@article{DW19,
author = {Dobrowolski, Jan and Wagner, Frank},
year = {2019},
pages = {},
title = {On Omega-categorical groups and rings of finite burden},
volume = {236},
journal = {Israel Journal of Mathematics},
doi = {10.1007/s11856-020-1989-9}
}

@article{DG19,
author = {Dobrowolski, Jan and Goodrick, John},
year = {2019},
pages = {},
title = {Some remarks on inp-minimal and finite burden groups},
volume = {58},
journal = {Archive for Mathematical Logic},
doi = {10.1007/s00153-018-0634-3}
}

@ARTICLE{Shelah14,
  title    = "Strongly dependent theories",
  author   = "Shelah, Saharon",
  journal  = "Israel Journal of Mathematics",
  volume   =  "204",
  number   =  "1",
  pages    = "1--83",
  year     =  {2014}
}

@book{Baldwin17, place={Cambridge}, series={Perspectives in Logic}, title={Fundamentals of Stability Theory}, publisher={Cambridge University Press}, author={Baldwin, John T.}, year={2017}, collection={Perspectives in Logic}}

@misc{Fujita25,
      title={Nonvaluational ordered Abelian groups of finite burden}, 
      author={Masato Fujita},
      year={2025},
      eprint={2502.18721},
      archivePrefix={arXiv},
      primaryClass={math.LO},
      note={preprint available at https://arxiv.org/abs/2502.18721}, 
}

@article{CS19, title={Henselian valued fields and inp-minimality}, volume={84}, DOI={10.1017/jsl.2019.56}, number={4}, journal={The Journal of Symbolic Logic}, author={Chernikov, Artem and Simon, Pierre}, year={2019}, pages={1510–1526}}

@article{Touchard23,
title = {Burden in Henselian valued fields},
journal = {Annals of Pure and Applied Logic},
volume = {174},
number = {10},
pages = {103318},
year = {2023},
issn = {0168-0072},
doi = {https://doi.org/10.1016/j.apal.2023.103318},
url = {https://www.sciencedirect.com/science/article/pii/S0168007223000751},
author = {Pierre Touchard}
}
\end{document}